\documentclass{article}
\usepackage[utf8]{inputenc}
\usepackage{lipsum}
\usepackage[finnish,swedish,english]{babel}
\usepackage{url}
\usepackage{graphicx}
\usepackage{amsmath}
\usepackage{amsthm}
\usepackage{amssymb}
\usepackage{mathrsfs}
\usepackage{setspace}
\usepackage{color}

\newtheorem{theorem}{Theorem}[section]
\newtheorem{lemma}[theorem]{Lemma}

\newtheorem{corollary}[theorem]{Corollary}

\newtheorem{example}[theorem]{Example}

\author{Erik Sj\"oland}
\title{Enumeration of monochromatic three term arithmetic progressions in two-colorings of cyclic groups}

\begin{document}

\newpage

\maketitle

\begin{abstract}
One of the toughest problems in Ramsey theory is to determine the existence of monochromatic arithmetic progressions in groups whose elements have been colored. We study the harder problem to not only determine the existence of monochromatic arithmetic progressions, but to also count them. We reformulate the enumeration in real algebraic geometry and then use state of the art computational methods in semidefinite programming and representation theory to derive sharp, or an explicit constant from sharp, lower bounds for the cyclic group of any order.
\end{abstract}

\tableofcontents

\newpage

\section{Introduction}

In Ramsey theory there are two closely related concepts of forcing structures, by colors or density: If we color the integers with a finite number of colors, then there are monochromatic arithmetic progressions of any length. If we choose any subset of the integers with positive density, then there are also arithmetic progressions of any length.

Getting finite, one can ask how large does $n$ need to be for us to find a monochromatic progression in $1,2,\ldots, n$, given a certain number of colors and a desired length of progressions? And more difficult than concluding the existence of monochromatic arithmetic progressions, could we count them? That is the goal of our research project.

To count the  monochromatic arithmetic progressions, we reformulate it as a problem of real algebraic geometry that can be attacked by state of the art optimization theory and provide human readable proofs. This is not the ordinary relaxation right off from binary combinatorial problems to linear or semidefinite programming, but one that allows us to use representation theory for modding out symmetries.

In the next section we present our results, and after that we present our methods and proofs.

\section{Results}
The case $n \mod 24 \in \{1,5,7,11,13,17,19,23\}$ in the following theorem was previously known \cite{Cameron2007}. To prove the theorem they used orthogonal arrays, and we do not see how their methods could be generalized to the other cases. Our method is based on completely different methods, and to our knowledge all other cases are new.

\begin{theorem}
\label{thm:cycliccase}
Let $n$ be a positive integer and let $R(3,\mathbb{Z}_n,2)$ denote the minimal number of monochromatic $3$-term arithmetic progressions in any two-coloring of $\mathbb{Z}_n$. $n^2/8-c_1n+c_2 \leq R(3,\mathbb{Z}_n,2) \leq n^2/8-c_1n+c_3$ for all values of $n$, where the constants depends on the modular arithmetic and are tabulated in the following table.
\[
\begin{array}{c|c|c|c}
n \mod 24 & c_1 & c_2 & c_3 \\
\hline
1,5,7,11,13,17,19,23 & 1/2 & 3/8 & 3/8 \\
8,16 & 1 & 0 & 0 \\
2,10 & 1 & 3/2 & 3/2 \\
4,20 & 1 & 0 & 2 \\
14,22  & 1 & 3/2 & 3/2 \\
3,9,15,21 &  7/6 & 3/8 & 27/8 \\ 
0 &  5/3 & 0  & 0 \\ 
12 &  5/3 & 0 & 18 \\ 
6,18 &  5/3 &1/2 & 27/2 \\ 
\end{array}
\]
\end{theorem}

It requires little work to see that the following result follows:

\begin{corollary}
\label{cor:dihedralcase}
Let $n$ be a positive integer. Let $R(3,\mathbb{Z}_n,2)$ and $R(3,D_{2n},2)$ denote the minimal number of monochromatic $3$-term arithmetic progressions in any two-coloring of $\mathbb{Z}_n$ and $D_{2n}$ respectively. The following equality holds
\[
R(3,D_{2n},2)=2R(3,\mathbb{Z}_n,2).
\]
In particular $n^2/4-2c_1n+2c_2 \leq R(D_{2n};3) \leq n^2/4-2c_1n+2c_3$ where the constants can be found in the table of Theorem \ref{thm:cycliccase}.
\end{corollary}

\section{Polynomial optimization}
\label{sec:poly}
The established results can be understood without understanding any polynomial optimization, but finding a sum of squares certificate by hand would be very difficult. In order to find the algebraic certificates in Theorem \ref{thm:cycliccase} results from real algebraic geometry had to be used to rewrite the problem of counting arithmetic progressions as a semidefinite program. Patterns were then found in the numerical solutions of the semidefinite programs, and through careful analysis of the numerical patterns we found the algebraic certificates that proves our main result. In this article we will only give some elementary definitions and examples for the unfamiliar reader that are sufficient for our approach. For the interested reader we refer to the extensive survey by Laurent \cite{Laurent2009}. We refer to \cite{Sjoland_methods} for a survey on how the methods were used and implemented for the particular problem.

A polynomial optimization problem is a problem on the form

\[
\begin{array}{rll}
\rho_* = \inf  & \displaystyle  f(x) \\
\textnormal{subject to} & \displaystyle  g_1(x) \geq 0, \dots, g_m(x) \geq 0, \\
& \displaystyle  x \in \mathbb{R}^n,
\end{array}
\]
where $f(x),g_1(x),\dots,g_m(x)$ are given polynomials. A strategy to solve a polynomial optimization problem is to introduce a new variable $\lambda$, and study the related problem:
\[
\begin{array}{rll}
\rho^* = \sup  & \lambda \\
\textnormal{subject to} & \displaystyle f(x)-\lambda \geq 0, g_1(x) \geq 0, \dots, g_m(x) \geq 0 \\
& \lambda \in \mathbb{R}, x \in \mathbb{R}^n
\end{array}
\]
where $f(x),g_1(x),\dots,g_m(x)$ are the same polynomials. For a discussion on the relationship between these problems we refer to the book by Lasserre \cite{Lasserre2010}, containing among other things the proof of strong duality; that $\rho_* = \rho^*$.

In real algebraic geometry one tries to find various relationships between nonnegative polynomials and sums of squares, which are known as Positivstellens\"atze. Let $X$ be a formal indeterminate, and define
\[
\begin{array}{rll}
\sigma^* = \sup  & \lambda \\
\textnormal{subject to} & \displaystyle f(X) - \lambda= \sigma_0 + \sum_{i=1}^m \sigma_ig_i \\
& \displaystyle  \sigma_i \textrm{ is a sum of squares in } X.
\end{array}
\]
It is easy to see that $\sigma^* \leq \rho^*$, and it can be proven that $\sigma^* = \rho^*$ under some technical conditions (Archimedean). This is known as Putinar's Positivstellensatz. The problem can be relaxed to $\sigma^*_d$, in which there are no monomials of degree larger than $d$. It is easy to see that $\sigma_{d_1}^* \leq \sigma_{d_2}^*$ if $d_1 < d_2$, and Lasserre \cite{Lasserre2001} has proven that if the Archimedean condition hold, then
\[
\lim_{d \rightarrow \infty} \sigma^*_d = \sigma^* = \rho^* = \rho_*.
\]
The reason we rewrote the positivity condition as a sum of squares condition is that the latter is equivalent to a semidefinite condition: $f(x)$ is a sum of squares of degree $2d$ if and only if it is possible to write $f(x) = v_d^T Q v_d$ where $v_d$ is the vector of all monomials up to degree $d$ and $Q$ is some positive semidefinite matrix. This makes it possible to find $\sigma^*_d$ and a sum of squares based certificate for the lower bound of our original polynomial, $f(x) = \sigma^*_d + \sigma_0 + \sum_{i=1}^m \sigma_i g_i \geq \sigma^*_d$, using semidefinite programming.

\begin{example}
To find a lower bound to
\[
\begin{array}{rll}
\rho_* = \inf  & \displaystyle  x_1x_2 \\
\textnormal{subject to} & \displaystyle  -1\leq x_1,x_2 \leq 1, \\
& \displaystyle  x \in \mathbb{R}^2,
\end{array}
\]
we write
\[
\begin{array}{rll}
\rho^* = \sup  & \lambda \\
\textnormal{subject to} & \displaystyle x_1x_2 - \lambda \geq 0, -1\leq x_1,x_2 \leq 1, 
\end{array}
\]
and look at the degree 3 relaxation, where $v_1=[1,X_1,X_2]^T$:
\[
\begin{array}{rll}
\sigma_2^* = \sup  & \lambda \\
\textnormal{subject to} & \displaystyle f(x) - \lambda= v_1^TQ_0v_1 + \sum_{i=1}^2 v_1^TQ_i^+v_1 (1+X_i)+ \sum_{i=1}^2 v_1^TQ_i^-v_1 (1-X_i), \\
& Q_0,Q_1^+,Q_2^+,Q_1^-,Q_2^- \succeq 0.
\end{array}
\]
Note that this is a semidefinite program with variables $\lambda, Q_0,Q_1^+,Q_2^+,Q_1^-,Q_2^-$, which we can solve using any software for semidefinite programming. A lower bound for $x_1x_2$ on the set $-1 \leq x_1,x_2 \leq 1$ is obtained by solving the relaxed optimization problem. We find a sum of squares based certificate:
\[ 
\begin{array}{rl}
X_1X_2 =&\displaystyle -1+\frac{1}{2}(X_1+X_2)^2 + \frac{1}{2}(1-X_1)^2(1+X_1) + \frac{1}{2}(1+X_1)^2(1-X_1) \\
               &\displaystyle+ \frac{1}{2}(1-X_2)^2(1+X_2) +\frac{1}{2}(1+X_2)^2(1-X_2) \\
               =&\displaystyle -1+ \frac{1}{2}(X_1+X_2)^2 + (1-X_1^2) + (1-X_2^2)  \\
               \geq &\displaystyle -1.
               \end{array}
\]
Since this is a simple example, it is easy to verify that this is the sharpest possible lower bounds using other methods. Examples with many variables and high degrees cannot be solved exact using other methods, and obtaining a lower bound using semidefinite programming is the current state-of-the art for many polynomial optimization problems.
\end{example}

\section{Counting monochromatic arithmetic progressions using semidefinite program}
\label{sec:methodscolor}
An arithmetic progression in $\mathbb{Z}_n$ of length $k$ is a $k$-set ($k$ distinct element) $\{a,a+b,\dots, a+(k-1)b \}$ where $a \in \mathbb{Z}_n$ and $b \in \mathbb{Z}_n \smallsetminus \{0\}$. When summing over all arithmetic progression we note that for example $\{1,2,3\}$, $\{1,3,2\}$, $\{2,1,3\}$, $\{2,3,1\}$, $\{3,1,2\}$, $\{3,2,1\}$ all denote the same set, and hence we only use one representative to avoid double counting. 

Let $\chi : \mathbb{Z}_n \rightarrow \{-1,1\}$ be a $2$-coloring of the group $\mathbb{Z}_n$, and let $x_g = \chi(g)$ for all $g \in \mathbb{Z}_n$. Let also $x=[x_0, \dots, x_{n-1}]$ denote the vector of all variables $x_g$. For $a,b,c \in \mathbb{Z}_n$, let us introduce the polynomial
\[
\begin{array}{rl}
p(x_a,x_b,x_c) &= \displaystyle \frac{(x_a+1)(x_b+1)(x_c+1)-(x_a-1)(x_b-1)(x_c-1)}{8}  \\
&= \displaystyle \frac{x_ax_b+x_ax_c+x_bx_c+1}{4}
\end{array}
\]
which has the property that
\[
p(x_a,x_b,x_c) = \left\{ 
\begin{array}{ll}
1 & \text{if }  x_a=x_b=x_c\\
0 & \text{otherwise.}
\end{array} \right. 
\]
In other words, if $\{a,b,c\}$ is an arithmetic progression, then the polynomial $p$ is one if $\{a,b,c\}$ is a monochromatic arithmetic progression and zero otherwise. It follows that the minimum number of monochromatic arithmetic progression of length 3 in a 2-coloring of $\mathbb{Z}_n$ is
\[
R(3,\mathbb{Z}_n,2)=  \min_{x \in \{-1,1\}^{n}} \displaystyle\sum_{\{a,b,c\} \textrm{ is an A.P. in } \mathbb{Z}_n}p(x_a,x_b,x_c).
\]

To find a lower bound for $R(3,\mathbb{Z}_n,2)$ we relax the integer quadratic optimization problem to a quadratic optimization problem on the hypercube. Since it is a relaxation, i.e. any solution of the integer program is also a solution to the hypercube problem, we have
\[
\begin{array}{rl}
R(3,\mathbb{Z}_n,2) &\displaystyle \geq \min_{ x \in [-1,1]^{n}} \sum_{\{a,b,c\} \textrm{ is an A.P. in } \mathbb{Z}_n}p(x_a,x_b,x_c) \\
& = \displaystyle \min_{ x \in [-1,1]^{n}} \sum_{\{a,b,c\} \textrm{ is an A.P. in } \mathbb{Z}_n} \frac{x_ax_b+x_ax_c+x_bx_c+1}{4} \\
& = \displaystyle \min_{ x \in [-1,1]^{n}} \frac{p_n}{4}+\sum_{\{a,b,c\} \textrm{ is an A.P. in } \mathbb{Z}_n}  \frac{1}{4}
\end{array}
\]
where 
\[ p_n = \sum_{\{a,b,c\} \textrm{ is an A.P. in } \mathbb{Z}_n}x_ax_b+x_ax_c+x_bx_c. \]
A priori we do not know how much we loose by doing the relaxation from the integer hypercube to the continuous hypercube. As it turns out for this problem we loose very little as the certificates in the main theorem are at most a constant from the solution we would get from the integer program.

Finding a lower bound for the homogeneous degree 2 polynomial $p_n$ immediately gives us a lower bound for $R(3,\mathbb{Z}_n,2)$. The monomial $x_ax_b$ occurs in $p_n$ as many times as the pair $(a,b)$ occurs in a 3-arithmetic progression, which depends on $n$. To find a lower bound to $ \min_{ x \in [-1,1]^{n}}  p_n$ it is suitable to use the state-of-the-art methods surveyed in Section \ref{sec:poly}.

In particular, let us use the degree 3 relaxation of Putinar's Positivstellensatz, and let the maximal lower bound using this relaxation be denoted $\lambda^*$. Let $v_1=[X_0, \dots, X_{n-1}, 1]^T$ be the vector of all monomials of degree less or equal to one. We get 
\[
\begin{array}{rl}
\lambda^* = \max  & \lambda \\
\textnormal{subject to:} & \displaystyle p_n - \lambda = v_1^T Q_0 v_1 + \sum_{i=0}^{n-1} v_1^T Q_i^+ v_1(1+X_i) + \sum_{i=0}^{n-1}  v_1^T Q_i^- v_1 (1-X_i), \\
& \displaystyle \lambda \in \mathbb{R}, \\
& \displaystyle Q_0,Q_i^+,Q_i^- \succeq 0 \textnormal{ for all } i \in \mathbb{Z}_n,
\end{array}
\]
which is a semidefinite program that can be solved numerically for fixed $n$.

The coefficients in the sum of squares certificate we get using these methods are numerical, and hence to get an algebraic positivity certificate one has to analyze the solutions further. There is no general way of doing this, and in most cases it is difficult to find good lower bounds using this procedure. We need algebraic certificates to provide solutions to all cyclic groups, not only for the cases we can solve numerically.

It was enough to solve the semidefinite program above for $n \leq 20$ before we found the patterns that lead to Theorem \ref{thm:cycliccase}. All the necessary computations could easily be carried out on a laptop within minutes to find numerical solutions. To extend the numerical solutions for $n \leq 20$ into algebraic certificates for all $n$ was more cumbersome and required more work. 

To go from numerical solutions to algebraic certificates is for most problems very difficult. For a specific group, all the information required to find a lower bound is contained in an eigenvalue decomposition of the involved matrices, but there is no general way of finding the optimal algebraic lower bound when the eigenvalues and eigenvectors have decimal expansions that cannot trivially be translated into algebraic numbers. If one is interested in a rational approximation to the lower bound one can use methods by Parrilo and Peyrl \cite{Parrilo2008}. These methods gives a nice certificate for a specific problem but provide little help when one want to find certificates for an infinite family of problems.

To find an algebraic certificate for all cyclic groups one of the tricks we used was to restrict the SDP above further without changing the optimal value. We required some entires to equal one another and forced some entries to be zero. There are also many other ways to restrict the SDP further that works for other problems. Another trick one can try is to change the objective function slightly to try to force the SDP to have only one optimal solution instead of infinitely many.

\section{Proofs of Theorems \ref{thm:cycliccase} and Corollary \ref{cor:dihedralcase}}
\label{sec:proof1}

To make the computations in the proofs that follow readable, let us introduce additional notation:
\[
\sigma(a;b_0,b_1,\ldots,b_{n-1})=a+ \sum_{i,j\in \mathbb{Z}_n} b_{j-i} X_iX_j.
\]

By elementary calculations we have the following equalities, which we need in the proofs:

\[
I_1=\displaystyle  \sum_{i\in \mathbb{Z}_{n}} (1-X_i^2)  =  \sigma(n;  -1,  0, \dots ,  0),
\]
\[
I_2=\displaystyle \Big( \sum_{i\in \mathbb{Z}_{n}} X_i \Big)^2  = \sigma\Big( 0;  1,  2,  \dots,  2 \Big), 
\]
\[
I_3=\displaystyle \Big( \sum_{i\in \mathbb{Z}_{n}} (-1)^i X_i \Big)^2 = \sigma\Big( 0;  1, -2, 2, -2, \dots, 2, -2 \Big), 
\]
\[
I_4=\displaystyle \sum_{i \in \mathbb{Z}_{n/2-1}}(X_i-X_{i+n/2})^2 = \sigma\Big( 0; 1 , 0 , \dots, 0 ,-2 , 0, \dots, 0 \Big) , 
\]
\[
\begin{array}{rl}
I_5 &= \displaystyle \sum_{i \in \mathbb{Z}_{n/3-1}}(X_i-X_{i+n/3})^2 +(X_i-X_{i+2n/3})^2   + (X_{i+n/3}-X_{i+2n/3})^2 \\
&=  \displaystyle \sigma\Big( 0; 2 , 0 , \dots, 0 ,-2 , 0, \dots, 0, -2, 0, \dots, 0 \Big) . 
\end{array}
\]

The first one is non-negative due to the boundary conditions $-1 \leq x_i \leq 1$ and the other ones are non-negative since they are sums of squares.

To prove Theorem \ref{thm:cycliccase} we need a small Lemma.
\begin{lemma}
\label{lem:even}
Let us consider $\mathbb{Z}_{n}$ for any positive integer $n$ such that $4$ divides $n+2$. The number of monochromatic $3$-arithmetic progressions is even for all $2$-colorings of $\mathbb{Z}_n$.
\end{lemma}
\begin{proof}
It is easy to see that the total number of arithmetic progressions is even.

Let $R \subseteq \{0,\dots, n-1\}$ be the red elements and $B= \{0,1,\dots,n-1\} \smallsetminus R$ be the blue elements in a coloring of $\mathbb{Z}_n$. Let $G(V,E)$ be the graph with vertices labeled by $\{1,\dots,n-1\}$ and edges $e_{ij}$ between the vertices $i$ and $j$ whenever $\{0,i,j\}$ forms an arithmetic progression.

Pick a vertex $a$ in $G$ and denote the vertices adjacent to $a$ by $\{n_1,\dots,n_d\}$. If $a=n/2$ it is easy to see that $d=0$, if $a$ is another odd labeled vertex we have $d=2$ since the integer $a$ is in the arithmetic progressions $\{n-a,0,a\}$ and $\{0,a,2a\}$ and if $a$ is in an even labeled vertex then $d=4$ since the integer $a$ is in the arithmetic progressions $\{-a,0,a\},\{0,a,2a\},\{0,a/2,a\},\{0,a/2+n/2,a\}$. In particular we note that $d$ is always even.

$|E|$ is even because there are an even number arithmetic progressions $\{n-a,0,a\}$ and for every arithmetic progression $\{0,a,2a\}$ there is an arithmetic progression $\{0,n-a,n-2a\}$.

Let again $a \in G$, and let $N_b$ be the number of blue adjacent vertices, $N_r$ the number of red adjacent vertices, $d=N_b+N_r$ and let $N_{mixed}$ denote the number of edges from a red vertex to a blue vertex in $G$. Suppose WLOG that $a$ is blue. Switching the color of $a$ from blue to red would imply that the number of edges from a red vertex to a blue vertex in $G$ changes to $N_{mixed}'=N_{mixed}+N_b-N_r$. $N_b-N_r$ is even (possibly negative) since $d$ is even, hence since $N_{mixed}$ is $0$ if all vertices are blue $N_{mixed}'$ must be even.

Let $N_{mono}$ be the number of monochromatic arithmetic progressions in $\mathbb{Z}_n$. Since there are an even number of edges from a red vertex to a blue vertex in $G$ and since $|E|$ is even the number of edges with same-colored endpoints is even. Thus switching the color of $0$ would imply that there are $N_{mono}'=N_{mono}+N_{new}$ monochromatic arithmetic progressions in $\mathbb{Z}_n$ where $N_{new}$ is even (possibly negative). The same argument holds if we change $0$ to any other vertex, and thus any change of $N_{mono}$ is by an even number. Since $N_{mono}$ is even if all of $\mathbb{Z}_n$ have the same color it follows by induction that $N_{mono}$ is even.
\end{proof}

\begin{proof}[Proof of Theorem \ref{thm:cycliccase}]
Recall that for $x \in \{ -1,1 \}^n$ we have
\[
R(3,\mathbb{Z}_n,2) =  \displaystyle\sum_{\{a,b,c\} \textrm{ is an A.P. in } \mathbb{Z}_n}\frac{x_ax_b+x_ax_c+x_bx_c+1}{4}.
\]
To find a lower bound $R(3,\mathbb{Z}_n,2)$ we need to find a lower bound for 
\[
p_n= \displaystyle\sum_{\{a,b,c\} \textrm{ is an A.P. in } \mathbb{Z}_n}x_ax_b+x_ax_c+x_bx_c.
\]

To express $p_n(X)$ in the $\sigma$-notation is the same as to count how many arithmetic progressions a pair $(a,b) \in \{ \mathbb{Z}_n \times \mathbb{Z}_n : a<b \}$ is in for all pairs $(a,b)$, which depends on the modular arithmetic of $n$. 

{\bf n mod 2 = 1, n mod 3 $\neq$ 0:}
It is easy to see that if $2$ and $3$ does not divide $n$, then $(a,b)$ is in exactly three different arithmetic progressions of $\mathbb{Z}_n$,
\[
 (a,b,c_1), (a,c_2,b), (c_3,a,b),
 \]
 and hence
\[
p_n(X) =  \sigma(0;0,3,\ldots,3).
\]

Since $p_n(X) = \frac{3}{2}I_2+\frac{3}{2}I_1 - \frac{3n}{2} \geq - \frac{3n}{2}$, and since there are $\binom{n}{2}$ distinct arithmetic progressions in $\mathbb{Z}_n$ we get
\[ 
R(3,\mathbb{Z}_n,2) \geq  \frac{\binom{n}{2}-\frac{3n}{2}}{4}=\frac{n^2-4n}{8} = \frac{(n-1)(n-3) - 3}{8}.
\]
Since $n$ is odd $8$ divides $(n-1)(n-3)$, and thus
\[
R(3,\mathbb{Z}_n,2) \geq \left\lceil \frac{(n-1)(n-3)-3}{8}  \right\rceil = \frac{(n-1)(n-3)}{8}
\]
since we know that the number of monochromatic arithmetic progressions is an integer.

The next step is to show that the lower bound is sharp by finding a coloring with $\frac{(n-1)(n-3)}{8}$ monochromatic arithmetic progressions. To achieve this we color $0,1,2,\ldots,(n-1)/2$  red and  $(n+1)/2,\ldots, n-2, n-1$  blue. All monochromatic progressions are given by their end points, and the end points should be of the same parity of the coloring. There are two cases, $n \equiv 1$ and  $n \equiv 3$ modulo 4, to treat separately. The different types of elements for the different cases are tabulated below:
\[
\begin{array}{c|cccc}
& \textrm{red even} & \textrm{red odd} & \textrm{blue even} & \textrm{blue odd} \\
\hline
n \equiv 1 (4) & (n+3)/4 & (n-1)/4 & (n-1)/4 & (n-1)/4 \\
n \equiv 3 (4) & (n+1)/4 & (n+1)/4 & (n+1)/4 & (n-3)/4 \\
\end{array}
\]
In both cases we get
\[
\begin{array}{rl}
\displaystyle R(3,\mathbb{Z}_n,2) & \displaystyle \leq { (n+3)/4 \choose 2 }+3{(n-1)/4 \choose 2} \\
&\displaystyle = 3{(n+1)/4 \choose 2} + { (n-3)/4 \choose 2 } \\
&\displaystyle = \frac{(n-1)(n-3)}{8}
\end{array}
\]
which shows that the bound is sharp.

{\bf n mod 8 = 0, n mod 3 $\neq$ 0:}

If $2$ divides $n$ and $b \neq a + n/2$ it follows that $(a,b)$ is in the $4$ pairs $(a,b,c_1)$, $(a,c_2,b)$, $(a,c_3,b)$, $(c_4,a,b)$ if $a-b$ is even and the $2$ pairs $(a,b,c_1)$, $(c_2,a,b)$ if $a-b$ is odd. When $b=a+n/2$ then $(a,a+n/2,a+n)=(a,b,a)$ and $(a-n/2,a,a+n/2)=(b,a,b)$ are degenerate arithmetic progressions which we do not count, hence
\[
p_n(X) =  \sigma(0;0,3,\ldots,3) +  \sigma(0;0,-1,1,-1,\ldots,1,-1) +  \sigma(0;0,\ldots,0,-2,0,\ldots,0).
\]

Since $p_n(X) = \frac{3}{2}I_2+\frac{1}{2}I_3+I_4+3I_1 - 3n \geq - 3n$, and since there are $\binom{n}{2}-\frac{n}{2}$ different (the term $-\frac{n}{2}$ comes from the fact that $(a,a+\frac{n}{2},a)$ is not an arithmetic progression for any $a$) arithmetic progressions in $\mathbb{Z}_n$ we get
\[ 
R(3,\mathbb{Z}_n,2)  \geq  \frac{\binom{n}{2}-\frac{n}{2}-3n}{4}=\frac{n^2-8n}{8}.
\]

To get an upper bound for $R(3,\mathbb{Z}_n,2)$ we present a coloring with as few monochromatic arithmetic progressions as possible. Partition $\mathbb{Z}_n$ into disjoint parts $G_1=\{0,\dots,n/4-1\}$, $G_2=\{n/4,\dots,n/2-1 \}$, $G_3=\{n/2, \dots, 3n/4-1\}$ and $G_4=\{3n/4,\dots,n-1\}$. Color $G_1$ and $G_3$ red, and $G_2$ and $G_4$ blue. Inside $G_1$ there will be $2i$ monochromatic arithmetic progressions of the form $(a,a+n/8-i,a+n/4-2i)$ for $1\leq i\leq n/8-1$, and hence a total number of $\sum_{i=1}^{n/8-1} 2i = 2\frac{\frac{n}{8}(\frac{n}{8}-1)}{2}=\frac{n^2}{64}-\frac{n}{8}$ arithmetic progressions. For any arithmetic progression $\{a,b,c\}$ in $G_1$, there is an arithmetic progression $\{a,b+n/2,c\}$ with $a,c$ in $G_1$ and $b+n/2$ in $G_3$. Since all elements of $G_1$ and $G_3$ have the same color, all the mentioned arithmetic progressions are monochromatic. By symmetry we get that the total number of monochromatic arithmetic progressions in the coloring is $8 \cdot (\frac{n^2}{64}-\frac{n}{8})$, hence
\[
R(3,\mathbb{Z}_n,2) \leq \frac{n^2-8n}{8}.
\]

{\bf n mod 8 = 2, n mod 3 $\neq$ 0:}
The arguments in the case n mod 8 = 0, n mod 3 $\neq$ 0 to show that $R(3,\mathbb{Z}_n,2) \geq \frac{n^2-8n}{8}$ requires only that $n$ is even, which holds also in this case. By Lemma \ref{lem:even} we can sharper the lower bound by rounding up to an even number;
\[ 
R(3,\mathbb{Z}_n,2) \geq \frac{(n-2)(n-6)}{8}.
\]
Let us partition $\mathbb{Z}_n$ into disjoint parts $G_1=\{0,\dots,(n+2)/4-1\}$, $G_2=\{(n+2)/4,\dots,n/2-1 \}$, $G_3=\{n/2,\dots,(3n+2)/4-1\}$ and $G_4=\{(3n+2)/4,\dots,n-1\}$. Color $G_1$ and $G_3$ red, and $G_2$ and $G_4$ blue. Note that $|G_1|=|G_3|=1+|G_2|=1+|G_4|$.

Inside $G_1$ there will be $2i$ monochromatic arithmetic progressions of the form $(a,a+(n+2)/8-i,a+(n+2)/4-2i)$ for $1\leq i\leq (n+2)/8-1$, and hence a total number of $\sum_{i=1}^{(n+2)/8-1} 2i = 2\frac{\frac{n+2}{8}(\frac{n+2}{8}-1)}{2}=\frac{(n+2)^2}{64}-\frac{n+2}{8}$ arithmetic progressions. 

Similarly, inside $G_1$ there will be $1+2i$ monochromatic arithmetic progressions of the form $(a,a+(n+2)/8-i,a+(n+2)/4-2i)$ for $0\leq i\leq (n+2)/8-2$, and hence a total number of $\sum_{i=0}^{(n+2)/8-2} (2i+1) =\frac{n+2}{8}-1 + 2\frac{(\frac{n+2}{8}-1)(\frac{n+2}{8}-2)}{2}$ arithmetic progressions.

To count all monochromatic arithmetic progressions in $\{1,\dots,n\}$ is equivalent to count $4\cdot \#(\textrm{ A.P. in $G_1$})+4\cdot \#(\textrm{ A.P. in $G_2$})$. Carrying out the elementary calculations we get
\[
R(3,\mathbb{Z}_n,2) \leq \frac{(n-2)(n-6)}{8}.
\]

{\bf n mod 8 = 4, n mod 3 $\neq$ 0:}
By the same arguments as in the case with n mod 8 = 0, n mod 3 $\neq$ 0 we get
\[
R(3,\mathbb{Z}_n,2) \geq \frac{n^2-8n}{8}.
\]

Let us partition $\mathbb{Z}_n$ into disjoint parts $G_1=\{0,1,\dots,n/4-1\}$, $G_2=\{n/4,\dots,n/2-1 \}$, $G_3=\{n/2,\dots,3n/4-1\}$ and $G_4=\{3n/4,\dots,n-1\}$. Color $G_1$ and $G_3$ red, and $G_2$ and $G_4$ blue. 

Inside $G_1$ there will be $2i-1$ monochromatic arithmetic progressions of the form $(a,a+(n+4)/8-i,a+(n+4)/4-2i)$ for $1\leq i\leq (n+4)/8-1$, and hence a total number of $\sum_{i=1}^{(n+4)/8-1} (2i-1) =-(\frac{n+4}{8}-1)+ 2\frac{\frac{n+4}{8}(\frac{n+4}{8}-1)}{2}=-\frac{n}{8}+\frac{4}{8}+\frac{n+4}{8}\frac{n-4}{8}=\frac{n^2}{64}-\frac{n}{8}+\frac{2}{8}$ arithmetic progressions.

To count all monochromatic arithmetic progressions in $\mathbb{Z}_n$ is equivalent to count $8\cdot \#(\textrm{ A.P. in $G_1$})$, hence
\[
R(3,\mathbb{Z}_n,2) \leq \frac{n^2-8n}{8} + 2.
\]

{\bf n mod 8 = 6, n mod 3 $\neq$ 0:}
The arguments in the case n mod 8 = 0, n mod 3 $\neq$ 0 to show that $R(3,\mathbb{Z}_n,2) \geq \frac{n^2-8n}{8}$ holds also in this case. By Lemma \ref{lem:even} we can sharper the lower bound to an even number;
\[ 
R(3,\mathbb{Z}_n,2) \geq \frac{(n-2)(n-6)}{8}.
\]

Let us partition $\mathbb{Z}_n$ into disjoint parts $G_1=\{0,1,\dots,(n+2)/4-1\}$, $G_2=\{(n+2)/4,\dots,n/2-1 \}$, $G_3=\{n/2,\dots,(3n+2)/4-1\}$ and $G_4=\{(3n+2)/4,\dots,n-1\}$. Color $G_1$ and $G_3$ red, and $G_2$ and $G_4$ blue. 

Inside $G_1$ there will be $2i-1$ monochromatic arithmetic progressions of the form $(a,a+(n+6)/8-i,a+(n+6)/4-2i)$ for $1\leq i\leq (n+6)/8-1$, and hence a total number of $\sum_{i=1}^{(n+6)/8-1} (2i-1) = -(\frac{n+6}{8}-1)+ 2\frac{\frac{n+6}{8}(\frac{n+6}{8}-1)}{2}=-\frac{n}{8}+\frac{2}{8} +\frac{n+6}{8}\frac{n-2}{8}
=
\frac{n^2}{64}-\frac{4n}{64}+\frac{4}{64}$ arithmetic progressions.

Similarly, inside $G_2$ there will be $2i$ monochromatic arithmetic progressions of the form $(a,a+(n+6)/8-i,a+(n+6)/4-2i)$ for $1\leq i\leq (n+6)/8-2$, and hence a total number of $\sum_{i=1}^{(n+6)/8-2} (2i) = 2\frac{(\frac{n+6}{8}-2)(\frac{n+6}{8}-1)}{2}=\frac{n-10}{8}\frac{n-2}{8}
=
\frac{n^2}{64}-\frac{12n}{64}+\frac{20}{64}$ arithmetic progressions.

To count all monochromatic arithmetic progressions in $\mathbb{Z}_n$ is equivalent to count $4\cdot \#(\textrm{ A.P. in $G_1$})+4\cdot \#(\textrm{ A.P. in $G_2$})$. Carrying out the elementary calculations we get

\[
R(3,\mathbb{Z}_n,2) \leq \frac{(n-2)(n-6)}{8}.
\]

{\bf n mod 2 = 1, n mod 3 = 0:}
If $3$ divides $n$ one can easily see that the only difference from when $3$ and $n$ are coprime is that the triples $(a,a+n/3,a+2n/3)$, $(a+2n/3,a,a+n/3)$ and $(a,a+2n/3,a+n/3)$ corresponds to the same arithmetic progression. In this case
\[
p_n(X) =  \sigma(0;0,3,\ldots,3) +  \sigma(0;0,\ldots,0,-2,0,\ldots,0,-2,0,\ldots,0).
\]

Since $p_n(X) = \frac{3}{2}I_2+I_5+ \frac{7}{2}I_1 - \frac{7}{2}n \geq - \frac{7}{2}n$, and since there are $\binom{n}{2}-\frac{2n}{3}$ different (the term $-\frac{2n}{3}$ comes from the fact that we do not want to count the triples $(a,a+\frac{n}{3},a+\frac{2n}{3})$ more than once) arithmetic progressions we get
\[ 
R(3,\mathbb{Z}_n,2) \geq  \frac{\binom{n}{2}-\frac{2n}{3}-\frac{7n}{2}}{4}=\frac{n^2}{8}-\frac{7n}{6}.
\]

To find a good coloring we split $\mathbb{Z}_n$ into disjoint parts $G_1,\dots,G_6$ with $|G_1|=|G_3|=|G_5|=|G_2|+1=|G_4|+1=|G_6|+1=\frac{n+3}{6}$ sorted such that if $i<j$ and we pick $a \in G_i$ and $b \in G_j$, then $a < b$. $G_1$,$G_3$ and $G_5$ are colored red, and $G_2$,$G_4$ and $G_6$ are colored blue.  All monochromatic progressions are given by their end points, and the end points should be of the same parity of the coloring. The way we partitioned the elements makes sure that every pair of monochromatic end points has a middle point in the same color. We start by considering pairs of end points in $G_1$ and $G_2$. We multiply this number by $9$ to get all pairs with $i<j$, $i$th element in $G_a$, $j$th element in $G_b$ for all $a-b \equiv 0(2)$, and then we add $\frac{n}{3}$ to get all arithmeric progressions of the type $(a,a+n/3,a+2n/3)$. There are two cases, $\frac{n+3}{6} \equiv 1$ and  $\frac{n+3}{6} \equiv 0 $ modulo 2, to treat separately. The different types of elements for the different cases are tabulated below:
\[
\begin{array}{c|cccc}
& \textrm{even in $G_1$} & \textrm{odd in $G_1$} & \textrm{even in $G_2$} & \textrm{odd in $G_2$} \\
\hline
(n+3)/6 \equiv 1 (2) & (n+9)/12 & (n-3)/12 & (n-3)/12 & (n-3)/12 \\
(n+3)/6 \equiv 0 (2) & (n+3)/12 & (n+3)/12 & (n+3)/12 & (n-9)/12 \\
\end{array}
\]
In both cases we get
\[
\begin{array}{rl}
\displaystyle R(3,\mathbb{Z}_n,2) & \displaystyle \leq 9({ (n+9)/12 \choose 2 }+3{(n-3)/12 \choose 2})+\frac{n}{3} \\
& \displaystyle = 9(3{(n+3)/12 \choose 2} + { (n-9)/12 \choose 2 })+ \frac{n}{3} \\
& \displaystyle = \frac{n^2}{8}-\frac{7n}{6}+\frac{27}{8}.
\end{array}
\]
Since this is an integer, and since we know that $R(3,\mathbb{Z}_n,2)$ is an integer, we can improve the lower bound slightly:
\[ 
R(3,\mathbb{Z}_n,2) \geq \frac{n^2}{8}-\frac{7n}{6}+\frac{3}{8}.
\]

{\bf n mod 8 = 0, n mod 3 = 0:}
By combining the arguments for when n mod 3 = 0 and when n mod 2 = 0 we find
\[
\begin{array}{c}
\displaystyle p_n(X) =  \sigma(0;0,3,\ldots,3) +  \sigma(0;0,-1,1,-1,\ldots,1,-1) +   \\
\displaystyle + \sigma(0;0,\ldots,0,-2,0,\ldots,0) +  \sigma(0;0,\ldots,0,-2,0,\ldots,0,-2,0,\ldots,0).
\end{array}
\]

Since $p_n(X) = \frac{3}{2}I_2+\frac{1}{2}I_3+I_4+I_5+5I_1 - 5n \geq - 5n$, and since there are $\binom{n}{2}-\frac{2n}{3}-\frac{n}{2}$ different arithmetic progressions we get
\[ 
R(3,\mathbb{Z}_n,2)\geq  \frac{\binom{n}{2}-\frac{2n}{3}-\frac{n}{2}-5n}{4}=\frac{n^2}{8}-\frac{5n}{3}.
\]

To find a good coloring we split $\mathbb{Z}_n$ into disjoint parts $G_1,\dots,G_{12}$ with $|G_1|=\dots =|G_{12}|=\frac{n}{12}$ sorted such that if $i<j$ and we pick $a \in G_i$ and $b \in G_j$, then $a < b$. Color $G_1,G_3,G_5,G_7,G_9,G_{11}$ red and $G_2,G_4,G_6,G_8,G_{10},G_{12}$ blue. By symmetry there will be equally many blue as red arithmetic progressions, so we restrict ourselves to counting red arithmetic progressions. Since $n$ is even, any arithmetic progressions in $\mathbb{Z}_n$ needs to have end points of the same parity. Let $e_2,e_2$ denote two elements of the same parity. If $e_1 \in G_1$ and $e_2 \in G_1 \cup G_5 \cup G_9$ then there are $2$ red arithmetic progressions $(e_1,a,e_2)$ and $(e_1,a+n/2,e_2)$, whereas if $e_1 \in G_1$ and $e_2 \in \mathbb{Z}_n \smallsetminus G_1 \cup G_5 \cup G_9$ there are no red monochromatic arithmetic progressions in $\mathbb{Z}_n$ with $e_1$ and $e_2$ as endpoints. Counting all possibilities one finds that if there are $m$ red arithmetic progressions in $G_1$ then there are $72m+\frac{n}{3}+n$ monochromatic arithmetic progressions in $\mathbb{Z}_n$. The term $72m$ is easily found using symmetries to find all arithmetic progressions "similar" to an arithmetic progression fully contained in $G_1$, $\frac{n}{3}$ comes from that all triples $(a,a+n/3,a+2n/3)$ are monochromatic, and the term $n$ comes from the arithmetic progressions of the type $(a,a+n/6, a+n/3)$.

To finish the calculation of the number of arithmetic progressions in the specified coloring we need to find the number of arithmetic progressions in $G_1$. We have:

\[
\begin{array}{c|cc}
& \textrm{even in $G_1$} & \textrm{odd in $G_1$}  \\
\hline
n/12 \equiv 0 (2) & n/24 & n/24 \\
\end{array}
\]
and thus we get
\[
\begin{array}{rl}
\displaystyle R(3,\mathbb{Z}_n,2) & \displaystyle \leq 2\cdot 72{ n/24 \choose 2 } +\frac{n}{3}+n\\
& \displaystyle = \frac{n^2}{8}-\frac{5n}{3}.
\end{array}
\]

{\bf n mod 8 = 4, n mod 3 = 0:}
As in the case when n mod 8 = 0, n mod 3 = 0 we get
\[ 
R(3,\mathbb{Z}_n,2)  \geq  \frac{\binom{n}{2}-\frac{2n}{3}-\frac{n}{2}-5n}{4}=\frac{n^2}{8}-\frac{5n}{3}.
\]

An equivalent analysis as for the case n mod 8 = 0, n mod 3 = 0 also shows that if there are $m$ red arithmetic progressions in $G_1$ then there are $72m+\frac{n}{3}+n$ monochromatic arithmetic progressions in $\mathbb{Z}_n$.

The calculation for the number of arithmetic progressions in $G_1$ is different than the previous case:

\[
\begin{array}{c|cc}
& \textrm{even in $G_1$} & \textrm{odd in $G_1$}  \\
\hline
n/12 \equiv 1 (2) & (n+12)/24 & (n-12)/24 \\
\end{array}
\]
and thus we get
\[
\begin{array}{rl}
\displaystyle R(3,\mathbb{Z}_n,2) & \displaystyle \leq 72{ (n+12)/24 \choose 2 }+ 72{ (n-12)/24 \choose 2 } + \frac{n}{3}+n  \\
& \displaystyle = \frac{n^2}{8}-\frac{5n}{3}+18.
\end{array}
\]

{\bf n mod 4 = 2, n mod 3 = 0:}
As in the case when n mod 8 = 0, n mod 3 = 0 we get
\[ 
R(3,\mathbb{Z}_n,2)  \geq  \frac{\binom{n}{2}-\frac{2n}{3}-\frac{n}{2}-5n}{4}=\frac{n^2}{8}-\frac{5n}{3}.
\]

Since $n$ is not divisible by $12$ we cannot pursue the problem identically to the case when n mod 8 = 0, n mod 3 = 0. We have to make a small adjustment and let $|G_1|=|G_3|=|G_5|=|G_7|=|G_9|=|G_{11}|=|G_2|+1=|G_4|+1=|G_6|+1=|G_8|+1=|G_{10}|+1=|G_{12}|+1=\frac{n+6}{12}$. When the same analysis as in that case we see that if there are $m_1$ red arithmetic progressions in $G_1$ and $m_2$ blue arithmetic progressions in $G_2$, then there are $36m_1+36m_2+\frac{n}{3}+n$ monochromatic arithmetic progressions in $\mathbb{Z}_n$.

To count the total number of arithmetic progressions note that we have the following two cases:
\[
\begin{array}{c|cccc}
& \textrm{even in $G_1$} & \textrm{odd in $G_1$} & \textrm{even in $G_2$} & \textrm{odd in $G_2$}  \\
\hline
(n+6)/12 \equiv 0 (2) & (n+6)/24 & (n+6)/24 & (n+6)/24 & (n-18)/24 \\
(n+6)/12 \equiv 1 (2) & (n+18)/24 & (n-6)/24 & (n-6)/24 & (n-6)/24 \\
\end{array}
\]
Thus it follows that
\[
\begin{array}{rl}
\displaystyle R(3,\mathbb{Z}_n,2) & \displaystyle \leq 36 \cdot 3{ (n+6)/24 \choose 2 }+ 36{ (n-18)/24 \choose 2 } + \frac{n}{3}+n  \\
& \displaystyle \leq 36 \cdot 3{ (n-6)/24 \choose 2 }+ 36{ (n+18)/24 \choose 2 } + \frac{n}{3}+n \\
& \displaystyle = \frac{n^2}{8}-\frac{5n}{3}+\frac{27}{2}.
\end{array}
\]
Since this is an integer and since $R(3,\mathbb{Z}_n,2)$ has to be an integer it follows that we can improve the lower bound slightly:
\[ 
R(3,\mathbb{Z}_n,2) \geq \frac{n^2}{8}-\frac{5n}{3}+\frac{1}{2}.
\]

\end{proof}

\begin{proof}[Proof of Corollary \ref{cor:dihedralcase}]
Let us denote the elements of the group $D_{2n}$ by $1$, $r$, $\dots$, $r^{n-1}$, $s$, $sr$, $\dots$, $sr^{n-1}$. Since the sets of elements $\{1,r,\dots,r^{n-1}\}$ and $\{s,sr,\dots,sr^{n-1} \}$ both contain the same number of arithmetic progressions as $\mathbb{Z}_n$ we get that $R(D_{2n};3) \geq 2R(\mathbb{Z}_n;3)+m$, where $m$ is the number of arithmetic progressions that does not only require rotation action. In $D_{2n}$ it holds that $rs=sr^{-1}$ and $s^2=1$, hence $(r^i,r^isr^j,r^isr^jsr^j)=(r^i,r^isr^j,r^issr^{-j}r^j)=(r^i,r^isr^j,r^i)$ is a degenerate arithmetic progression for any choice of $i$ and $j$. The same happens for arithmetic progressions $(sr^i,sr^isr^j,sr^isr^jsr^j)$. Since these are all possible arithmetic progressions of $D_{2n}$ that are not rotations it follows that $m=0$. Since all arithmetic progressions containing reflections are degenerate we can color the sets $\{1,r,\dots,r^{n-1}\}$ and $\{s,sr,\dots,sr^{n-1} \}$ independently, and thus $R(D_{2n};3)=2R(\mathbb{Z}_n;3)$ as desired.
\end{proof}

\section*{Acknowledgements}
I would like to thank Alexander Engstr\"om for introducing me to this problem and for his advice when I didn't know how to proceed. I also want to thank Markus Schweighofer and Cynthia Vinzant for their valuable feedback and corrections.


\begin{thebibliography}{99} 

\bibitem{Cameron2007}
Peter Cameron, 
Javier Cilleruelo and 
Oriol Serra. 
On monochromatic solutions of equations in groups. 
\emph{Rev. Mat. Iberoam.} {\bf 23} (2007), no. 1, 385--395.

\bibitem{Lasserre2001}
Jean Bernard Lasserre. 
Global optimization with polynomials and the problem of moments.
\emph{SIAM J. Optim.} {\bf 11} (2001), no. 3, 796--817.

\bibitem{Lasserre2010}
Jean Bernard Lasserre.
\emph{Moments, positive polynomials and their applications.}
Imperial College Press Optimization Series, 1. \emph{Imperial College Press, London,} 2010. 361 pp.

\bibitem{Laurent2009}
Monique Laurent. 
Sums of squares, moment matrices and optimization over polynomials.
Chapter of "Emerging applications of algebraic geometry". \emph{Springer, New York,} 2009. 157--270.

\bibitem{Parrilo2008}
Pablo A. Parrilo and 
Helfried Peyrl.
Computing sum of squares decompositions with rational coefficients.
\emph{Theoret. Comput. Sci.} {\bf 409} (2008), no. 2, 269--281.

\bibitem{Sjoland_methods}
Erik Sj\"oland.
Using real algebraic geometry to solve combinatorial problems with symmetries.
Preprint available at http://arxiv.org/abs/1408.1065.

\end{thebibliography}
\end{document}